 \documentclass[11pt]{article}
 \usepackage[margin=1in]{geometry}
 \usepackage{enumerate}
 \usepackage{amsmath}
 \usepackage[english]{babel}
 \usepackage{graphicx}
 \usepackage{color}
 \usepackage{setspace}
 \usepackage{multirow}
 \usepackage{float}
 \usepackage{amsthm}
 \usepackage{amsmath}
 \usepackage[all]{xy}
 \usepackage{amssymb}
 \usepackage{pigpen}
 \usepackage{url}
 \usepackage{caption}
 \usepackage{subcaption}
 \usepackage{graphicx}
 \usepackage{tikz, wrapfig,array}
 \usetikzlibrary{calc,patterns,
 	decorations.pathmorphing,
 	decorations.markings, shapes.geometric, graphs, graphs.standard, quotes,shapes,chains,scopes,positioning,arrows}
 \newcommand{\po}{\ar@{}[dr]|{\text{\pigpenfont R}}} \newcommand{\pb}{\ar@{}[dr]|{\text{\pigpenfont J}}}
 \newcommand{\Hom}{{\rm Hom}}
 \newcommand{\conn}{\rm conn}
 
 \def\p3{\mathbf{P}}
 \def\cG{\mathcal{G}}

 \def\C{\mathcal{C}}
 
 \def\x{\times}
 
 \def\max{\text{max}}
 \newtheorem{thm}{Theorem}[section]
 
 \newtheorem{defn}[thm]{Definition}
 \newtheorem{lemma}[thm]{Lemma}
 \newtheorem{prop}[thm]{Proposition}
 \newtheorem{cor}[thm]{Corollary}
\theoremstyle{definition}
\newtheorem{remark}[thm]{Remark}
 
 \def\N{\text{N}}
  \def\NC{\mathcal{N}}
 \def\S{\text{S}}
 \def\D{{\mathbb{D}}}
 \def\diam{\text{diam}}

\begin{document}
 	\title{ Lov\'{a}sz' original lower bound: Getting tighter bounds and Reducing computational complexity}
 	\date{}
 	\author{Shuchita Goyal\footnote{The author would like to thank CSIR, India and IRCC, IIT Bombay for their financial support.} \  and Rekha Santhanam} 	\maketitle
 	\begin{abstract}
 		In this article, we give conditions on a graph under which the Lov\'{a}sz' original bound of the graph can be improved by increasing the topological connectivity of its neighbourhood complex.
        We also work out conditions under which computing the topological connectivity of  hom complex of a pair of graphs can be simplified. 
        In particular, hom complex as a covariant functor acting on a double mapping cylinder of graphs is a homotopy pushout of  hom complex functor applied to its subgraphs. 
        We give applications of  this result where the computation of hom complexes is simplified. 
        Finally, we explain why double mapping cylinder of graphs does not give a satisfactory definition of homotopy pushout in the category of graphs.
			\end{abstract}
 	
 	\section{Introduction}
 	For a graph $G$, the question of determining the chromatic number of $G$ has been of central importance in graph theory. In the year 1955, Kneser formulated a class of graphs, popularly known as Kneser graphs, $KG_{n,k}$, and conjectured that the chromatic number of $KG_{n,k}$ is $n-2k+2$. This question remained open for almost 25 years until 1978 when Lov\'{a}sz came up with a topological space associated to a graph $G$, called the neighbourhood complex, $\NC(G)$, of that graph. His proof opened a wide area of mathematics, topological combinatorics. Lov\'{a}sz proved the Kneser's conjecture using topological properties of $\NC(KG_{n,k})$. In particular, he showed that $\chi(G)$ is bounded from below by the topological connectivity of $\NC(G)$ plus 3, for every graph $G$. This bound is often referred to as the Lov\'{a}sz'  original bound. Generalizing this idea, he also conjectured that for any two graphs $G,T$, the chromatic number of the graph $G$ is bounded from below as per the following inequality:  
 	
 	\begin{equation}\label{lovasz conjecture}
 		\chi(G) \ge \chi(T) + \conn \Hom(T,G) + 1.
 	\end{equation}
 	
  	A graph, $T$, for which this inequality holds with respect to every graph $G$ is called a test graph. In view of Kozlov's result in \cite{kozlov-hom-from-k2-is-neighbourhood-complex} which states that $\Hom(K_2,G)$ is homotopy equivalent to $\NC(G)$, Lov\'{a}sz result \cite{lov-kneser} showed that $K_2$ is a test graph.  It has been shown that the class of complete graphs \cite{BabKoz-HOMcplx}, odd cycles \cite{BabKoz-topo-Obstructions}, bipartite graphs \cite{Matsushita-bipartite}, special Kneser graphs \cite{kneser-ht-graph}, etc, are all test graphs. Lov\'{a}sz conjecture has been disproved by Hoory and Linial \cite{countereg-to-lov} who exhibited explicit graphs $G$ and $T$ such that the inequality \eqref{lovasz conjecture} fails to satisfy. 	
  	
 	In order to apply inequality \eqref{lovasz conjecture}, we need to be able to compute the connectivity of the hom complex.  
 	Some special  cases have been studied \cite{homintokn,homofcycles,homcycletokn}.  This computation, however, in general is extremely difficult. 
 	For instance, for $n$ greater than 3, computing whether $\Hom(T,K_n)$ is path connected is a PSPACE-complete problem \cite{pspace-complete}.

  	We begin this article with  methods to improve Lov\'{a}sz' original lower bound for chromatic number by replacing the given graph $G$ with a smaller graph.  More precisely, we give conditions on a degree 2 vertex $v$ of a graph $G$ so that $\chi(G-v) = \chi(G)$, and the connectivity of the neighbourhood complex of $G-v$ is greater than or equal to that of the  neighbourhood complex of $G$. We also analyse the conditions on a degree 3 vertex $v$ for which  the connectivity of the neighbourhood complex of $G-v$ is greater than or equal to  the connectivity of the neighbourhood complex of $G$. 
 	
     Next we study pushouts in graphs and their behaviour with respect to the $\Hom(T,\underline{ \ \ })$ functor for any graph $T$. 
     It is a quick check that in general pushouts are not preserved by the hom functor. 
     It turns out though that double mapping cylinders of graphs are preserved under $\Hom(T,\underline{ \ \ })$ functor. 
     For a fixed homotopy test graph $T$, and given any $m$ large enough, we use the properties of the double mapping cylinder to construct a graph $G$ whose chromatic number is $m$ and the topological connectivity of the hom complex, $\Hom(T,G)$, is 0.
     We also compute the homotopy type of neighbourhood complexes of a certain class of graphs.
 	
 Based on our applications it is useful to know when a given graph is $\x$-homotopy equivalent to the double mapping cylinder of some smaller graphs. Understanding double mapping cylinder of graphs as their homotopy pushout is then vital. In  Section \ref{sec:Dn not good notion section},  we point out why  double mapping cylinder of graphs cannot be a homotopy pushout object in the category of graphs. In a following article, we study this notion further.

	\section{Improving Lov\'{a}sz' original bound }\label{sec:choosing subgraphs}

	A {\it graph} $G$ is a pair of sets $(V(G),E(G))$. The elements of $V(G)$ are called {\it vertices} of $G$, and  the elements of $E(G)$, which are  two element subsets of $V(G)$, are called {\it edges} of $G$. These elements need not be distinct, that is, $\{y,y\}$ can also be an edge and this makes $y$ into a {\it looped} vertex. 
	We  say that two vertices  $x$, $y$ are adjacent if  $\{x,y\} \in E(G)$ and denote the edge by $xy$.  A  {\it simple graph} is a graph without any looped vertex. A {\it complete graph} on $n$ vertices is a simple graph where any two distinct vertices are adjacent to each other, and is denoted by $K_n$.	For $n \in \mathbb{N}$, let $I_n$ be the graph with vertex set $\{0,1,\dots,n\}$ and edge set $\{ij : |i-j| \le 1\}$. 
	
	A {\it homomorphism} between two graphs $f: G \to H$ is a function from $V(G)$ to $V(H)$ such that if $xy$ is an edge in $G$, then $f(x)f(y)$ is an edge in $H$. The {\it chromatic number} for any simple graph $G$ is defined to the least integer $n$ such that there is a graph homomorphism from $G $ into $K_n$. 
	
	\begin{defn}
		Let $v,v' \in G$ and $v \ne v'$. Then $v$ is said to fold to $v' \in G$ if every neighbour of $v$ is also a neighbour of $v'$. In this case, we say that $G$ folds to $G-v$. A graph $G$ is called stiff if there does not exist any $ v \in G$ such that $G$ folds to $G-v$.
	\end{defn}
	
	\begin{defn}
		Let $G , H \in \cG$ be two graphs. The (categorical) product of $G$ and $H$ is defined to be the graph $G \x H$ whose vertex set is the cartesian product, $V(G) \x V(H)$, and $(g,h)(g',h') \in E(G \x H)$ whenever $gg' \in E(G),\ hh' \in E(H) $.
	\end{defn}
	
	Let $\cG$ be the category with objects as finite undirected graphs without multiple edges and  morphisms as vertex set maps that takes edges to edges. 
		
	\begin{defn}
		[\cite{fold-in-hom}]For two arbitrary graphs $T$ and $G$, the hom complex, $\Hom(T,G)$, is the polyhedral complex whose cells are indexed by all the functions  $\eta : V(T) \rightarrow 2^{V(G)} \setminus \{\emptyset\}$ such that $\eta(x) \x \eta(y) \subset E(G)$, whenever $x y \in E(T)$. The closure of a cell $\eta$ consists of all cells $\eta'$, satisfying $\eta'(v) \subset \eta(v)$, for all $v \in V(G)$.
		
	\end{defn}
	\noindent Throughout this paper, we will use the notation of polytopes to mean cells of $\Hom(T,G)$. For a polytope $\eta \in \Hom(T,G)$, the image of $\eta$ is defined to be, $Im(\eta) := \displaystyle\cup_{t \in V(T)} \eta(t)$.
	In this section, we always consider connected loopless graphs unless stated otherwise. 

	For a given graph $G$ and a vertex $v \in V(G)$, we let $\N_G(v)$ denote the set of neighbours of $v$ in $G$, and for a set $A \subset V(G)$, $\N_G(A) = \bigcup_{a \in A} \N_G(v)$. Whenver there is no scope of confusion, the subscript from the notation of neighbourhood set is dropped. {\it Neighbourhood complex} of graph $G$, $\NC(G)$, is the simplicial complex whose simplices are formed by those subsets of vertex set of $G$ that have a common neighbour.
	
 	By definition of neighbourhood complex of a graph $G$, two vertices  $u,v \in V(G)$ form a 1-simplex, denoted by $\{u,v\}$, in $\NC(G)$ if and only if $u$ and $v$ have a common neighbour in $G$. That is, there exists a path of length 2 connecting $u$ to $v$ in $G$. Therefore, two vertices lie in the same path connected component of $\NC(G)$ if and only if there exists an even length path in $G$ connecting these two vertices. Following this idea, it is clear that all the vertices of an odd cycle of $G$ belong to the same path connected component of $\NC(G)$.
 	
 	\begin{lemma} \cite{lov-kneser} \label{bipartite-disconn charac}
 		A loopless connected graph  $G$ is bipartite if and only if $\NC(G)$ is disconnected.
 	\end{lemma}
 	
 	Let $G$ be any graph and $v \in V(G)$. The {\it degree of vertex} $v$ is defined as the cardinality of neighbourhood set of $v$, $|\N_G(v)|$ if $v$ is not looped and  $|\N_G(v)| + 1$ otherwise. The degree of $v$ is denoted as $d(v)$.
 	
 	\begin{thm}\label{fundamental grp comparison}
 		Let $G$ be a simple non-bipartite graph with $C_p$, a cycle in $G$ and $c \in C_p $ be such that $d(c) = 2$, $b \in \N_G(c)$. Then $\pi_1(\NC(G-bc),c) \text{ is a subgroup of }\pi_1(\NC(G),c).$
 	\end{thm}
 
 	In order to prove Theorem \ref{fundamental grp comparison}, we first establish some results about homotopy groups of simplicial complexes.
 	
	For a topological space $X$, $\pi_0(X)$ denotes the path connected components of $X$. Let $X$ and $Y$ be two topological spaces and $f,g : X\to Y$ be two continuous functions. The map $f$ is said to be homotopic to $g$ if there exists a continuous function $F : X \x [0,1] \to Y$ satisfying $F_{|X \x \{0\}} = f$ and  $F_{|X \x \{1\}} = g$. Let $\mathbb{S}^n$ be the $n$- dimensional sphere, then it is the boundary of $(n+1)$-dimensional disc, $\mathbb{D}^{n+1}$. For a topological space $X$ with a base point $x$, the $n^{th}$-homotopy group $\pi_n(X,x)$ is defined to be the collection of homotopy classes of maps from $\mathbb{S}^n \to X$. One often writes $[X,Y]$ to mean the homotopy class of maps from the space $X$ to $Y$. If a continuous function $f: \mathbb{S}^n \to X$ can be extended continuously to $\tilde{f} : \mathbb{D}^{n+1} \to X$ then $f$ is called homotopically trivial. The $n^{th}$-homotopy group of $X$ is trivial if every $f: \mathbb{S}^n \to X$ is homotopically trivial. The space $X$ is called $k$-connected if $\pi_n(X,x)$ is a trivial group  for all $0 \le n \le k$ and the connectivity of $X$, $\conn X$ is the largest $k$ such that $X$ is $k$-connected.

	Let $S$ be a  geometric simplicial complex. We define a 1-chain in $S$ is a finite sequence of 0-simplices $v_0 v_1 \dots v_s$ such that $\{v_i,v_{i+1}\}$ is a 1-simplex in $S$ for all $i = 0, \dots, s-1$. A 1-chain is closed if $v_0 = v_s$ and such closed 1-chains are also referred to as loops based at $v_0$. 
	
	A constant loop based at $v_0$ is a sequence $v_0 v_1 \dots v_s$ such that $v_i = v_0$ for all $i = 0, \dots,s $. A loop based at $v_0$ is said to be homotopically trivial in $S$ if it is homotopic to the constant loop based at $v_0$. We define a closed 1-chain as simple if $v_0$ occurs exactly once, that is, $v_i \ne v_0$ for $0 < i < s$.

	For a simplicial complex $S$, the n-skeleton, $S_n$, of $S$ is the subcomplex of $S$ that contains all $k$-simplices of $S$, for $k \le n$. 
	
	\begin{defn}
		Let $S$ and $T$ be two geometric simplicial complexes. A continuous map $f : S \to T$ is called a simplicial map if for any $k$-simplex $\sigma = \{s_0,s_1, \dots, s_k\}$ of $S$, $f(\sigma)$ is a simplex $\{f(s_0),f(s_1), \dots, f(s_k)\}$ of $T$.
	\end{defn}

	For an $n$-simplex $\sigma = \{x_0,x_1, \dots, x_n\}$ of a simplicial complex $X$, we define a point $x \in \sigma$ to be an interior point of $\sigma$ if $x$ is not one of the $x_i$, for $i = 0, 1, \dots, n$. A point $x \in X$ is called an interior point of $X$ if it is an interior point of some simplex of $X$.

	\begin{lemma}\label{contracting extension map from S2}
		Let $S$ be a geometric simplicial complex and $\alpha : \mathbb{S}^1 \to S$ be a closed nullhomotopic 1-chain in $S$. If the extension $\sigma: \mathbb{D}^2 \to S$, of $\alpha$, is such that $Im (\sigma) \subset S_1$, that is, $\sigma$ is contained in the 1-skeleton of $S$, then every interior point of $\alpha$ is repeated at least once.
	\end{lemma}
	
	\begin{proof}
		We prove it by contradiction. Considering the 2-ball $\mathbb{D}^2$ as a simplicial complex, by simplicial approximation theorem, there exists a simplicial map $\sigma_s : Y \to S$, where $Y$ denotes the $n$-th barycentric subdivision of  $\mathbb{D}^2$ for a large enough $n$. Then, $\sigma_s$ is homotopic to $\sigma$, and  $\sigma_s|_{bd(Y)}$ gives a simplicial map $\alpha_s : bd(Y) \to S$ such that $\alpha_s$ is homotopic to $\alpha$.
		
		Since $S$ is a simplicial complex, the intersection of any two simplices of $S$ is again a simplex of $S$ and a face of both the intersecting simplices. Therefore, if no 1-simplex of $\alpha_s$ is repeated, then at least three consecutive 1-simplices of $\sigma_s$ are unrepeated. Without any loss of generality, we can assume that $\alpha_s$ contains a subdivision of a hollow triangle, and hence $\sigma_s$ contains a subdivision of a hollow triangle. Let $\eta$ be a minimal length closed 1-chain in $Im(\sigma_s)$ containing a subdivision of a hollow triangle.
		
		Since $\sigma_s$ is a simplicial map, image of any two intersecting simplices under $\sigma_s$ intersects. Let $A_1$ be the inverse image of $\eta$ under $\sigma_s$, then $A_1$ contains intersecting 2-simplices that form a annulus like structure (please refer to Figure \ref{D2 to S1 map contradiction}). As a simplicial complex in itself, $A_1$ has a boundary that does not enclose any 2-simplex of $A_1$, say $\alpha_1$.
		\begin{figure}[ht] 
			\tikzstyle{black1}=[circle, draw, fill=black!100, inner sep=0pt, minimum width=4pt]
			\tikzstyle{ver}=[]
			\tikzstyle{extra}=[circle, draw, fill=black!00, inner sep=0pt, minimum width=2pt]
			\tikzstyle{edge} = [draw,thick,-]
			\tikzstyle{light} = [draw, gray,-]
			\centering
			
			\begin{tikzpicture}[scale=0.7] 
			
				\begin{scope}[shift={(-5,0)}]
					\draw[-, black!50,fill=gray!40] (-.5,1) -- (-.8,3) -- (0,2) -- (-.5,1);
					\draw[-, black!50,fill=gray!40] (0,2) -- (2,2) -- (1,1) -- (0,2);
					\draw[-, black!50,fill=gray!40] (2,2) -- (2.5,1) -- (1,1) -- (2,2);
					\draw[-, black!50,fill=gray!40] (2.5,1) -- (2,-1) -- (1,-1) -- (2.5,1);
					\draw[-, black!50,fill=gray!40] (1,-1) -- (.5,-2) -- (0,-1) -- (1,-1);
					\draw[-, black!50,fill=gray!40] (0,-1) -- (.5,-2) -- (-.5, -2) -- (0,-1);
					\draw[-, black!50,fill=gray!40] (0,-1) -- (-.5,-2) -- (-1, -1.5) -- (0,-1);
					\draw[-, black!50,fill=gray!40] (-1,-1.5) -- (-1,0) -- (-2,-1) -- (-1,-1.5);
					\draw[-, black!50,fill=gray!40] (-1,0) -- (-1.8,1) -- (-.5,1) -- (-1,0);
					\draw[-, black!90] (0,2) -- (1,1) -- (2.5,1) -- (1,-1) -- (0,-1) -- (-1,-1.5) -- (-1,0) -- (-.5,1) -- (0,2);
					
					\node[below] at (0,-2.5) {$(a)$};
				\end{scope}
				
				\begin{scope}[shift={(5,0)}]
					\draw[-, black!50,fill=gray!40] (-.5,1) -- (-.8,3) -- (0,2) -- (-.5,1);
					\draw[-, black!50,fill=gray!40] (0,2) -- (2,2) -- (1,1) -- (0,2);
					\draw[-, black!50, fill=gray!40] (2,2) -- (2.5,1) -- (1,1) -- (2,2);
					\draw[-, black!50,fill=gray!40] (2.5,1) -- (2,-1) -- (1,-1) -- (2.5,1);
					\draw[-, black!50,fill=gray!40] (1,-1) -- (.5,-2) -- (0,-1) -- (1,-1);
					\draw[-, black!50,fill=gray!40] (0,-1) -- (.5,-2) -- (-.5, -2) -- (0,-1);
					\draw[-, black!50,fill=gray!40] (0,-1) -- (-.5,-2) -- (-1, -1.5) -- (0,-1);
					\draw[-, black!50,fill=gray!40] (-1,-1.5) -- (-1,0) -- (-2,-1) -- (-1,-1.5);
					\draw[-, black!50,fill=gray!40] (-1,0) -- (-1.8,1) -- (-.5,1) -- (-1,0);
										
				
					\draw[-, black!50,fill=gray!40] (-1,0) -- (0,0) -- (-.5,1);
					\draw[-, black!50,fill=gray!40] (-.5,1) -- (0,0) -- (0,2);
					\draw[-, black!50,fill=gray!40] (0,2) -- (0,0) -- (1,1);
					\draw[-, black!50,fill=gray!40] (1,1) -- (.5,0) -- (2.5,1);
					\draw[-, black!50,fill=gray!40] (.5,0) -- (1,-1) -- (2.5,1);
					\draw[-, black!50,fill=gray!40] (0,-1) -- (.5,-.5) -- (1,-1);
					\draw[-, black!50,fill=gray!40] (0,-1) -- (-.5,-.5) -- (-1,-1.5);
					\draw[-, black!50,fill=gray!40] (-1,-1.5) -- (-.5,-.5) -- (-1,0);
					
					\draw[-,black!90] (0,2) -- (1,1) -- (2.5,1) -- (1,-1) -- (0,-1) -- (-1,-1.5) -- (-1,0) -- (-.5,1) -- (0,2);
					
					\draw[edge] (-1,0) -- (0,0)-- (1,1) -- (.5,0) -- (1,-1) -- (.5,-.5) -- (0,-1) -- (-.5, -.5) -- (-1,0);
					\node[below] at (0,-2.5) {$(b)$};
				\end{scope}
			\end{tikzpicture}
			\caption{\footnotesize The shaded region in $(a)$ denotes an annulus like structure $A_1$ described in the proof of Lemma \ref{contracting extension map from S2}, with dark line illustrating the boundary $\alpha_1$. The figure in $(b)$ shows the next iteration where $A_2$ is obtained  from $A_1$ of $(a)$, with darkest line showing the boundary $\alpha_2$.}
			\label{D2 to S1 map contradiction}
		\end{figure}
		Let $A_2$ be the collection of 2-simplices of $Y-A_1$ with a facet belonging to $\alpha_1$. For any 2-simplex $\tau \in A_2$, the value of $\sigma_s$ on its facet determines the value of $\tau$, and $\sigma_s$ being a simplicial map implies that $\sigma_s(\tau) \in \eta$. Note that $A_2$ also forms an annulus like structure with possibly few pinches along the two boundaries. Let $\alpha_2$ be the boundary of $A_2$ away from the boundary of $Y$. We repeat the same argument for $A_2$, and so on. Since $Y$ is a finite simplicial complex, the process terminates after finitely many iterations with a 2-simplex with two of its facets being mapped to different 1-simplices of $\eta$, a contradiction that $\sigma_s$ is a simplicial map.
	\end{proof}

	Similar arguments can be used to prove a more general statement, that is:
	
	\begin{prop}  \label{contracting extension map}
		Let $S$ be a geometric simplicial complex and $\alpha : \mathbb{S}^n \to S$ be a nullhomotopic continuous map such that for any extension $\sigma: \mathbb{D}^{n+1} \to S$, of $\alpha$, $Im (\sigma) \subset S_n$, that is, $\sigma$ is contained in the n-skeleton of $S$, then every interior point of facets in $\alpha$ is repeated at least once.
	\end{prop}

	\begin{proof}[Proof of Theorem \ref{fundamental grp comparison}]
		By hypothesis, $G$ is non-bipartite and hence $\NC(G)$ is connected by Lemma \ref{bipartite-disconn charac}. Let $\N_G(c) = \{a,b\}$. If $c$ belongs to a $C_4= (a,c,b,d)$ in $G$, then $\N_G(c) = \{a,b\} \subset \N_G(d)$. Therefore $c$ folds to $d$ and $\NC(G-c)$ is homotopy equivalent to $\NC(G)$ and the result then follows. 
		
		We, therefore,  assume that the vertex $c$ does not fold in $G$ and $c $ does not belong to any $C_4$ in $G$. Let $\N_G(a) = \{a_1,a_2, \dots, a_k, c\}$ and $\N_G(b) = \{b_1,b_2, \dots, b_r, c\}$. Let $A$ and $B$ denote the neighbour set of $a$ and $b$ in $G$ respectively. Since $c$ is a degree 2 vertex that belongs to a cycle, its neighbours $a$ and $b$ also belong to the same cycle and hence $k,r \ge 1$. From now on, we drop the subscript $G$ from $\N_G$.
		Since $c \in V(G)$ does not fold  in $G$, $(A \cap B) - \{c\} = \emptyset$.

		Consider the graph $G' = G - bc$. 
		Since $G'$ is a subgraph of $G$, $\NC(G')$ is a subcomplex of $\NC(G)$. Let $f : \NC(G') \to \NC(G)$ be the inclusion map and $f_{\#} : \pi_1(\NC(G'),c) \to \pi_1(\NC(G),c)$ be the group homomorphism induced by $f$. To establish the statement, it suffices to prove that $f_{\#}$ is an injective homomorphism.
				
		Let $\alpha \in \pi_1(\NC(G'),c)$ be a closed simple 1-chain such that $f_{\#}(\alpha)$ is homotopically trivial in $\NC(G)$. Then $\alpha = c v_1 v_2 \dots v_t c $ is a sequence of vertices of $G'$ such that the first and the last vertex is $c$ and any two consecutive vertices form a 1-simplex in $\NC(G')$ and hence have a common neighbour in $G'$. 
		Since in $\NC(G')$, the vertex $c$ forms 1-simplex only with neighbours of $a$, namely, $\{a_1, a_2, \dots, a_k\}$, we can see that $v_1, v_t \in \{a_1, a_2, \dots, a_k\}$. 
		
		Using simplicial approximation, a closed 1-chain in a simplicial complex $X$ can be interpreted as a simplicial map from a triangulation of $\mathbb{S}^1$ to $X$. Given that $\alpha : \mathbb{S}^1 \to \NC(G)$ is homotopically trivial in $\NC(G)$, there exists an extension $\sigma_{\alpha} : \mathbb{D}^2 \to \NC(G)$, that is $bd(\sigma_{\alpha}) = \alpha$. In the context of simplicial complex, the subcomplex $\sigma_{\alpha} \subset \NC(G)$ gives a triangulation of $\mathbb{D}^2$. If $\sigma_{\alpha}$ is a subcomplex of $\NC(G')$ then it gives an extension of $\alpha$ to $\mathbb{D}^2$ in $\NC(G')$ and $\alpha$ is homotopically trivial in $\NC(G')$ and there is nothing to prove. Also, if $\sigma_{\alpha}$ belongs to 1-skeleton of $\NC(G)$, then the result follows from Proposition \ref{contracting extension map}. So without loss of generality, assume that $\sigma_{\alpha} $ contains a 2-simplex $\{c,b_1,b_2\}$ 
		
		Note that any 2-simplex of $\NC(G)$ not containing $c$ belongs to $\NC(G')$. If the 2-simplex $\{c,b_1,b_2\}$ belongs to the interior of $\sigma_{\alpha}$, then there exists a closed 1-chain $\eta$ homotopic to $\alpha$ in $\NC(G')$ such that boundary of $\sigma_{\eta}$ is $\eta$, and $\eta$ contains the 1-simplex $\{b_1,b_2\}$.
		Therefore, we assume that if a simple closed 1-chain $\alpha \subset \NC(G')$ is homotopically trivial in $\NC(G)$ and $\sigma_{\alpha}$ contains a 2-simplex $\{c,b_1,b_2\}$, then the 1-simplex $\{b_1,b_2\}$ belongs to $\alpha$. 
		
		\begin{figure}[ht]
			\tikzstyle{black1}=[circle, draw, fill=black!100, inner sep=0pt, minimum width=4pt]
			\tikzstyle{ver}=[]
			\tikzstyle{extra}=[circle, draw, fill=black!00, inner sep=0pt, minimum width=2pt]
			\tikzstyle{edge} = [draw,thick,-]
			\tikzstyle{light} = [draw, gray,-]
			\centering
			
			\begin{tikzpicture}[scale=0.5]
			
			\begin{scope}[shift={(-15,0)}]

			\draw[edge,black] (3,0) -- (1,-1) -- (2.5,-1.5);
			\draw[light] (-3,0) -- (1,-1);
			 
			\node[below] at (0.8,-1) {$c$};
			
			\node[right] at (2.5,1.8) {$v_2$};		
			\node[right] at (1.5,2.8) {$v_3$};							
			\node[left] at (-2.5,1.8) {$v_{i-1}$};										
			\node[left] at (-3,0) {$b_1$};			
			\node[left] at (-2.5,-1.8) {$v_{i+1}$};														
			\node[right] at (3,0) {$a_1$};		
			\node[right] at (2.5,-1.8) {$a_2$};

			\draw[light] (-3,0) -- (-1,0.3) -- (0,1) -- (1,0) -- (1,-1);
			\draw[light] (-2.6,1.5) -- (-1,0.3) -- (1,-1);
			\draw[light] (-2.6,1.5) -- (0,1) -- (1,-1) ;
			\draw[light] (1,0) -- (2,1) -- (3,0);
			\draw[light] (1.5,2.6) -- (2,1) -- (2.6,1.5);
			\draw[light] (-2.6,1.5) -- (0,3) -- (0,1) -- (1.5,2.6) -- (1,0) ;
			
			
			\draw[light] (1,-1) -- (1.5,-2.5) -- (-1,-1.5) -- (-2.5,-1.5);
			\draw[light] (1,-1) -- (-1,-1.5) -- (-3,0);
			\draw[light] (0,-3) -- (-1,-1.5) -- (-1.45,-2.5);
			\draw[light] (1,0) -- (3,0);
			\draw [fill] (-1,-1.5) circle [radius=0.15];						
			\draw [fill] (-1,0.3) circle [radius=0.15];
			\draw [fill] (0,1) circle [radius=0.15];
			\draw [fill] (1,0) circle [radius=0.15];
			\draw[fill] (2,1) circle [radius=0.15];
			
			\foreach \x/\y in 				{0/v_{1},30/v_{2},60/v_{3},90/v_{4},120/v_{5},150/v_{6},180/v_{7},210/v_{8},240/v_{9},270/v_{10},300/v_{11},330/v_{12}}{
				\node[black1] (\y) at (\x:3){};
			}
			
			\node[black1] (v) at (1,-1){};
			
			\foreach \x/\y in {v_{1}/v_{2},v_{2}/v_{3},v_{3}/v_{4},v_{4}/v_{5},v_{5}/v_{6},v_{6}/v_{7},v_{7}/v_{8},v_{8}/v_{9},v_{9}/v_{10},v_{10}/v_{11},v_{11}/v_{12}}{
				\path[edge] (\x) -- (\y);}
			
			\foreach \x/\y in {v_{1}/v_{2},v_{2}/v_{3},v_{3}/v_{4},v_{4}/v_{5},v_{5}/v_{6},v_{6}/v_{7},v_{7}/v_{8},v_{8}/v_{9},v_{9}/v_{10},v_{10}/v_{11},v_{11}/v_{12}}{
				\path[edge,black] (\x) -- (\y);}
			
			\node[below] at (0,-4) {(a)};
				
			\end{scope}

			
			\begin{scope}[shift={(-7,0)}]
			\foreach \x/\y in {0/v_{1},30/v_{2},60/v_{3},90/v_{4},120/v_{5},150/v_{6},180/v_{7},210/v_{8},240/v_{9},270/v_{10},300/v_{11},330/v_{12}}{
				\node[black1] (\y) at (\x:3){};
			}
			
			\node[black1] (v) at (1,-1){};
			\foreach \x/\y in {v_{1}/v_{2},v_{2}/v_{3},v_{3}/v_{4},v_{4}/v_{5},v_{5}/v_{6},v_{6}/v_{7},v_{7}/v_{8},v_{8}/v_{9},v_{9}/v_{10},v_{10}/v_{11},v_{11}/v_{12}}{
				\path[light] (\x) -- (\y);}
			\foreach \x/\y in {v_{1}/v,v_{7}/v,v_{12}/v}{
			\path[light] (\x) -- (\y);}
			
		\foreach \x/\y in {v_{1}/v_{2},v_{2}/v_{3},v_{3}/v_{4},v_{4}/v_{5},v_{5}/v_{6},v_{6}/v_{7}}{
			\path[edge] (\x) -- (\y);}
			\draw[edge] (3,0) -- (1,-1) -- (-3,0);
			\node[below] at (0.8,-1) {$c$};

			\node[right] at (2.5,1.8) {$v_2$};		
			\node[right] at (1.5,2.8) {$v_3$};							
			\node[left] at (-2.5,1.8) {$v_{i-1}$};										
			\node[left] at (-3,0) {$b_1$};			
			\node[left] at (-2.5,-1.8) {$v_{i+1}$};														
			\node[right] at (3,0) {$a_1$};		
			\node[right] at (2.5,-1.8) {$a_2$};		
			\draw[light] (1,0) -- (3,0);
			\node [below] at (-1,0.45) { $b_{i_1}$};
			\draw [fill] (-1,0.3) circle [radius=0.15];
			\draw [fill] (0,1) circle [radius=0.15];
			\draw [fill] (1,0) circle [radius=0.15];
			\draw[fill] (2,1) circle [radius=0.15];
			
			\draw[light] (-3,0) -- (-1,0.3) -- (0,1) -- (1,0) -- (1,-1);
			\draw[light] (-2.6,1.5) -- (-1,0.3) -- (1,-1);
			\draw[light] (-2.6,1.5) -- (0,1) -- (1,-1) ;
			\draw[light] (1,0) -- (2,1) -- (3,0);
			\draw[light] (1.5,2.6) -- (2,1) -- (2.6,1.5);
			\draw[light] (-2.6,1.5) -- (0,3) -- (0,1) -- (1.5,2.6) -- (1,0) ;
				
			
			\draw [fill] (-1,-1.5) circle [radius=0.15];			
			\draw[light] (1,-1) -- (1.5,-2.5) -- (-1,-1.5) -- (-2.5,-1.5);
			\draw[light] (1,-1) -- (-1,-1.5) -- (-3,0);
			\draw[light] (0,-3) -- (-1,-1.5) -- (-1.45,-2.5);
			\node[below] at (0,-4) {(b)};			
			\end{scope}

			
			\begin{scope}[shift={(1,0)}]
			\foreach \x/\y in {0/v_{1},30/v_{2},60/v_{3},90/v_{4},120/v_{5},150/v_{6},180/v_{7},210/v_{8},240/v_{9},270/v_{10},300/v_{11},330/v_{12}}{
				\node[black1] (\y) at (\x:3){};
			}
			
			\node[black1] (v) at (1,-1){};
			\foreach \x/\y in {v_{1}/v_{2},v_{2}/v_{3},v_{3}/v_{4},v_{4}/v_{5},v_{5}/v_{6},v_{6}/v_{7},v_{7}/v_{8},v_{8}/v_{9},v_{9}/v_{10},v_{10}/v_{11},v_{11}/v_{12}}{
				\path[light] (\x) -- (\y);}
			\foreach \x/\y in {v_{1}/v,v_{7}/v,v_{12}/v}{
				\path[light] (\x) -- (\y);}
			
			\foreach \x/\y in {v_{1}/v_{2},v_{2}/v_{3},v_{3}/v_{4},v_{4}/v_{5},v_{5}/v_{6},v_{6}/v_{7}}{
				\path[edge] (\x) -- (\y);}

			\node[below] at (0.8,-1) {$c$};

			\node[right] at (2.5,1.8) {$v_2$};		
			\node[right] at (1.5,2.8) {$v_3$};							
			\node[left] at (-2.5,1.8) {$v_{i-1}$};										
			\node[left] at (-3,0) {$b_1$};			
			\node[left] at (-2.5,-1.8) {$v_{i+1}$};														
			\node[right] at (3,0) {$a_1$};		
			\node[right] at (2.5,-1.8) {$a_2$};		
			
			\node [below] at (-1,0.45) { $b_{i_1}$};
			\draw [fill] (-1,0.3) circle [radius=0.15];
			\node [above] at (-0.43,1) { $b_{i_2}$};
			\draw [fill] (0,1) circle [radius=0.15];
			\draw [fill] (1,0) circle [radius=0.15];
			\draw[light] (1,0) -- (3,0);			
			\draw[fill] (2,1) circle [radius=0.15];
			
			\draw[light] (-3,0) -- (-1,0.3) -- (0,1) -- (1,0) -- (1,-1);
			\draw[light] (-2.6,1.5) -- (-1,0.3) -- (1,-1);
			\draw[light] (-2.6,1.5) -- (0,1) -- (1,-1) ;
			\draw[light] (1,0) -- (2,1) -- (3,0);
			\draw[light] (1.5,2.6) -- (2,1) -- (2.6,1.5);
			\draw[light] (-2.6,1.5) -- (0,3) -- (0,1) -- (1.5,2.6) -- (1,0) ;
			
			\draw[edge] (3,0) -- (1,-1) -- (-1,0.3) -- (-3,0);
				
			
			\draw [fill] (-1,-1.5) circle [radius=0.15];			
			\draw[light] (1,-1) -- (1.5,-2.5) -- (-1,-1.5) -- (-2.5,-1.5);
			\draw[light] (1,-1) -- (-1,-1.5) -- (-3,0);
			\draw[light] (0,-3) -- (-1,-1.5) -- (-1.45,-2.5);
			\node[below] at (0,-4) {(c)};			
			\end{scope}

			
			\begin{scope}[shift={(9,0)}]
			\foreach \x/\y in {0/v_{1},30/v_{2},60/v_{3},90/v_{4},120/v_{5},150/v_{6},180/v_{7},210/v_{8},240/v_{9},270/v_{10},300/v_{11},330/v_{12}}{
				\node[black1] (\y) at (\x:3){};
			}
			
			\node[black1] (v) at (1,-1){};
			\foreach \x/\y in {v_{1}/v_{2},v_{2}/v_{3},v_{3}/v_{4},v_{4}/v_{5},v_{5}/v_{6},v_{6}/v_{7},v_{7}/v_{8},v_{8}/v_{9},v_{9}/v_{10},v_{10}/v_{11},v_{11}/v_{12}}{
				\path[light] (\x) -- (\y);}
			\foreach \x/\y in {v_{1}/v,v_{7}/v,v_{12}/v}{
				\path[light] (\x) -- (\y);}
			\foreach \x/\y in {v_{1}/v_{2},v_{2}/v_{3},v_{3}/v_{4},v_{4}/v_{5},v_{5}/v_{6},v_{6}/v_{7}}{
				\path[edge] (\x) -- (\y);}
			
			\node[below] at (.8,-1) {$c$};

			\node[right] at (2.5,1.8) {$v_2$};		
			\node[right] at (1.5,2.8) {$v_3$};							
			\node[left] at (-2.5,1.8) {$v_{i-1}$};										
			\node[left] at (-3,0) {$b_1$};			
			\node[left] at (-2.5,-1.8) {$v_{i+1}$};														
			\node[right] at (3,0) {$a_1$};		
			\node[right] at (2.5,-1.8) {$a_2$};		
			
			\node [below] at (-1,0.45) { $b_{i_1}$};
			\draw [fill] (-1,0.3) circle [radius=0.15];
			\node [above] at (-0.43,1) { $b_{i_2}$};
			\draw [fill] (0,1) circle [radius=0.15];
			\node [right] at (1,-0.3) { $b_{i_j}$};
			\draw [fill] (1,0) circle [radius=0.15];
			
			\node [left] at (2,1) {$x$};
			\draw[fill] (2,1) circle [radius=0.15];
			\draw[dotted,thick] (1,0) -- (3,0);			
			\draw[light] (-3,0) -- (-1,0.3) -- (0,1) -- (1,0) -- (1,-1);
			\draw[light] (-2.6,1.5) -- (-1,0.3) -- (1,-1);
			\draw[light] (-2.6,1.5) -- (0,1) -- (1,-1) ;
			\draw[light] (1,0) -- (2,1) -- (3,0);
			\draw[light] (1.5,2.6) -- (2,1) -- (2.6,1.5);
			\draw[light] (-2.6,1.5) -- (0,3) -- (0,1) -- (1.5,2.6) -- (1,0) ;
			
			\draw[edge] (3,0) -- (1,-1) -- (0,1) -- (-1,0.3) -- (-3,0);
				
			
			\draw [fill] (-1,-1.5) circle [radius=0.15];			
			\draw[light] (1,-1) -- (1.5,-2.5) -- (-1,-1.5) -- (-2.5,-1.5);
			\draw[light] (1,-1) -- (-1,-1.5) -- (-3,0);
			\draw[light] (0,-3) -- (-1,-1.5) -- (-1.45,-2.5);
			\node[below] at (0,-4) {(d)};			
			\end{scope}
			
			\end{tikzpicture}
			
			\caption{\footnotesize An example to illustrate the proof of the Theorem \ref{fundamental grp comparison}. In (a), $\alpha$ is shown in black as the boundary of $\sigma_{\alpha}$. In (b), $\alpha_1$ is shown in black as obtained from $\alpha$ such that $\sigma_{\alpha_1}$ is contained in $\sigma_{\alpha}$. Similarly (c) records the next step. Lastly, (d) shows  the final iteration, where we ought to have $\{c,b_{i_j},a_1\}$ as a 2-simplex shown by a dotted line.}
		\end{figure}
		
		Let $i$ be the smallest integer such that $v_i \in \{b_1, b_2, \dots , b_k\}$, so that the simple closed 1-chain $\alpha = c a_1 v_2 v_3 \dots v_{i-1} v_i v_{i+1} \dots v_{t-1} a_2 c$ with $v_i = b_1$, say,  is homotopically trivial in $\NC(G)$ but not in $\NC(G')$. 
		We will show that this is not possible by exhibiting a simple closed 1-chain in $\sigma_{\alpha}$ which is not homotopically trivial in $\NC(G)$. 
		
		Take $\alpha_1 = c a_1 v_2 \dots v_{i-1} b_1 c$, then $\alpha_1$ is a subcomplex of $\sigma_{\alpha}$ and hence also homotopically trivial in $\NC(G)$. Let $ \sigma_{\alpha_1}$ denote the extension of $\alpha_1$ to $\mathbb{D}^2$. 
		Note that $\N(A) \cap \N(B) = \emptyset$, thus $\{c, a_i, b_j\}$ cannot be a 2-simplex in the neighbourhood complex of $G$. If $\{c,b_1,x\}$ is a 2-simplex, then $x \in \{b_1, \dots , b_k\}$. 
		Therefore $\alpha$ being homotopically trivial implies that there exists $b_{i_1} \in \{b_1, \dots , b_k\} $ such that 2-simplex $\{c, b_1, b_{i_1}\} \subset \sigma_{\alpha_1} $. Now consider $\alpha_2 = c a_1 v_2 \dots b_1 b_{i_1} c \subset \sigma_{\alpha_1}$, then $\alpha_2$ is also homotopically trivial in $\NC(G)$. Similar argument as above gives that there exists $b_{i_2} \in \{b_1, \dots , b_k\} $ such that 2-simplex $\{c, b_{i_1}, b_{i_2}\} \subset \sigma_{\alpha_2}$. Arguing similarly, we get $\alpha_j = c a_1 v_2 \dots b_1 b_{i_1} b_{i_2} \dots b_{i_j} c \subset \sigma_{\alpha_{j-1}} \subset \sigma_{\alpha}$ such that $\alpha_j$ is also homotopically trivial in $\NC(G)$. Since $G$ is finite and hence the set $\{b_1, \dots , b_k\}$ of neighbours of $b$, this process  terminates in a finite number of steps. 
		During this process, all the $v_i$'s gets replaced by elements from $A\cup B$ and we get a closed 1-chain  $\{c, a_i,b_j\}$, which is not homotopically trivial in $\NC(G)$, and that is a contradiction. Therefore, if $\alpha$ is homotopically trivial in $\NC(G)$, then  it is homotopically trivial in $\NC(G')$ also.	
	\end{proof}	
	
\noindent Note that $G'$ considered in the above theorem folds to $G - c$.	Since $c$ belongs to $C_n$ in $G$, $G' = G-bc$ is connected for a degree 2 vertex $c$ of $G$, if $\chi(G-c) > 2$, then $\chi(G) = \chi(G-c)$.
 	
 	\begin{cor}\label{G and G' equivalence}
 		For a finite connected simple non-bipartite graph, if $c \in V(G)$ belongs to a cycle in $G$ and $d(c) = 2$. Then 
 		$$ \conn \NC(G-c) \ge \conn \NC(G).$$
 		Furthermore, if $G-c$ contains an odd cycle then $\chi(G-c) = \chi(G)$. Else, $\chi(G-c) = 2$ implying $\chi(G) = 3$.		
		
 	\end{cor}
	\begin{proof}
	To establish the claim, we need to only show that if $\pi_k(\NC(G))$ is trivial for $k\geq 2$, then $\pi_k(\NC(G-bc))$ is trivial.   	Let $\alpha : \S^n \to \NC(G-c) $ be a simplicial map for $n\geq 2$. Then by assumption we know it extends to a simplicial map $\alpha': \D^{n+1} \to \NC(G)$. Since we assume that $a,b$ have only one vertex $c$ in common, we know that the  $1$-simplex $\{a,b\}$ is a part of $\NC(G)$ but not a face of higher boundary in $\NC(G)$ and does not exist in $\NC(G-bc)$. If image of $\alpha'$ includes this $1$-simplex $\{a,b\}$, then by Proposition \ref{contracting extension map} we can see that we can get a new map $\alpha'' : \D^{n+1} \to \NC(G)-\{a,b\} $. Now it is possible that $\N_G(b)$ is in the image of $\alpha'$. Let  the simplex ${c}\cup \N_{G-c}(b)$, be in the image of $\alpha'$. Since image of $\alpha$ does not intersect with any simplex in ${c}\cup \N_{G-c}(b)$ containing $c$, we can contract the simplices containing $c$ down to $\N_{G-c}(b)$. Thus defining an $\alpha'': D^{n+1}\to \NC(G-c)$.  Hence proved. 
 	\end{proof}
	    Let $G$ be a graph, then its minimum degree, $\delta(G)$, is defined to be the minimum of degrees taken over all the vertices of $G$. Let $\C^{(2)}$ be the class of graphs whose minimum degree is 2. 	Then Theorem \ref{fundamental grp comparison} can be applied to any  non-bipartite graph $G \in \C^{(2)}$. Note that $G$ does not have any hanging (degree 1) vertices. Also, since the minimum degree, $\delta(G) = 2$, every vertex belongs to some cycle in the graph.  Let $v \in V(G)$ be such that $d(v) = 2$. Since $v$ belongs to a cycle, Theorem \ref{fundamental grp comparison} applies to it, and we get the graph $G-v$ such that $\pi_1(N(G-v))$ is a subgroup of $\pi_1(N(G))$, and either $G-v$ is bipartite or $\chi(G-v) = \chi(G)$.

\vspace{0.2cm}

Let $(G,c) =(G_0,c_0)$ be a connected non-bipartite graph with vertex $c_0 \in V(G_0)$ of degree two that belongs to some cycle $C_p \subset G$. Define $G_1$ to be the subgraph $G_0 - \{c_0\}$,  where $G_0 - \{c_0\}$ denotes the subgraph of $G_0$ induced by the vertex set $V(G_0) - \{c_0\}$.  Then by Corollary \ref{G and G' equivalence}, $\conn \NC(G_1) \ge \conn \NC(G_0)$. Since $G_0$ is non-bipartite, $\chi(G_0) \ge 3$. If $G_1$ does not contain any odd cycle then it is bipartite and hence has chromatic number 2, thereby implying that the chromatic number of $G_0$ is 3 and we stop. On the other hand,  if $G_1$ contains an odd cycle, then again by Corollary \ref{G and G' equivalence},  $\chi(G_1)= \chi(G_0)$. In this case, if $G_1$ contains a degree two vertex say $c_1$, then we repeat as above to obtain a smaller graph $G_2 = G_1 - \{c_1\}$; and as before if $G_2$ is bipartite then $2 = \chi(G_2) = \chi(G_1) - 1 = \chi(G_0) - 1$ implying that $\chi(G_0 = 3)$. Otherwise $\chi(G_2) = \chi(G_1) = \chi(0)$.

We continue removing degree two vertices from non-bipartite graph as explained above. Since the original graph $G_0$ is finite, the number of iterations is also finite. Suppose $n$ vertices are removed during this iterative process, then using Lov\'{a}sz' bound, we have  $$\chi(G) = \chi(G_0) = \chi(G_n) \ge \conn \NC(G_n) + 3.$$ Since $G_n$ is a much smaller graph compared to that of $G_0$, computation of $\conn \NC(G_n)$ is expected to be easier than $\conn \NC(G)$.

\vspace{0.2cm}

The result in Theorem \ref{fundamental grp comparison} deals with the fundamental group of the neighbourhood complex of the subgraph $G-c$ as compared to that of $G$. In the cases when $\pi_1(\NC(G))$ is non-trivial but $\pi_1(\NC(G-c))$ is, it improves the Lov\'{a}sz' lower bound  by at most 1. In the same vein, we prove the next result where removing degree $3$ vertices  can improve the bound by $2$ in good cases.

 	\begin{thm} \label{higher homotopy grp comparison}
 		Let $G$ be a non-bipartite graph with a degree three non-looped vertex $v \in V(G)$. Then
 		\begin{enumerate}
 			\item The second homotopy group $\pi_2(\NC(G-v))$ is a subgroup of $\pi_2(\NC(G))$.
 			
 			\item If for every pair of two neighbours $x,y$ of $v$, there exists an even length path joining $x$ and $y$ in $G-v$, then 
 			$$\pi_0(\NC(G-v)) = \pi_0(\NC(G)).$$
 			Furthermore, for some $x,y \in \N_G(v)$, let $P = x p_1 p_2 \dots p_{2n-1} y$ be a shortest distance path from $x$ to $y$ in $G-v$. If $n \ge 2$, then 
 			$$\pi_1(\NC(G-v)) \text{  is a subgroup of } \pi_1(\NC(G)).$$ 
 		\end{enumerate} 
 	\end{thm}
 	
 	\begin{proof}
 		Let the neighbours, $\N_G(v)$, of $v$ be $\{a,b,c\}$. Without loss of generality, we can assume that $\N_G(a) \cap \N_G(b) \cap \N_G(c) = \{v\}$, otherwise if there exists $x \in \N_G(a) \cap \N_G(b) \cap \N_G(c)$, $x \ne v$, then $v$ folds to $x$. By \cite{kozlov-hom-from-k2-is-neighbourhood-complex,fold-in-hom}, $\NC(G-v)$ is homotopy equivalent to $\NC(G)$, and hence the result holds. 
 		
 		Since $G-v$ is a subgraph of $G$, $\NC(G-v)$ is a subcomplex of $\NC(G)$. Let $f : \NC(G-v) \to \NC(G)$ denote this inclusion, and $f_{\#}: \pi_2(\NC(G-v)) \to \pi_2(\NC(G))$ be the induced group homomorphism.
 		
 		Let $\alpha : \mathbb{S}^2 \to \NC(G-v)$ be a continuous map such that $f_{\#}(\alpha) = f \alpha$ is nullhomotopic. To prove the first part, it suffices to show that $\alpha$ is nullhomotopic. Since $f \alpha : \mathbb{S}^2 \to \NC(G)$ is nullhomotopic, there exists an extension $\sigma : \mathbb{D}^3 \to \NC(G)$. The 3-ball $\mathbb{D}^3$ carries a simplicial complex structure and $\NC(G)$ is a simplicial complex. Therefore, by simplicial approximation theorem, there exists a simplicial map $\sigma_s : Y \to \NC(G)$, where $Y$ denote the n-th barycentric subdivision of $\mathbb{D}^3$ for some large enough $n$. Let $\alpha_s$ be the restriction of $\sigma_s$ to the boundary, $bd(Y)$, of $Y$.
 		
 		If $Im(\sigma_s) $ is contained in $\NC(G-v)$, then $\alpha$ is trivial in $\NC(G-v)$. Therefore, assume the $Im(\sigma_s)$ contains a simplex of $\NC(G)$ that is not in $\NC(G-v)$. Also, if $Im(\sigma_s)$ is contained in the 2-skeleton of $ \NC(G)$, then the result follows from Proposition \ref{contracting extension map}. So, assume that $Im(\sigma_s)$ contains a 3-simplex of $\NC(G)$ that is not in $\NC(G-v)$.
 		
 		Let $A := \N_G(a)$, $B := \N_G(b)$, $C := \N_G(c)$. If $\{v,x,y,z\}$ forms a 3-simplex of $\NC(G)$ that is not in $\NC(G-v)$, then all three $x,y,z$ belong to either $A$, or $B$, or $C$. Since $\sigma|_{bd(\mathbb{D}^3)} = \alpha$ is contained in $\NC(G-v)$, $v$ is an interior point of $Im(\sigma)$. If the 3-simplex generated by $\{v,x,y,z\}$ lies in the interior of $Im(\sigma)$, then nullhomotopicity of $f \alpha$ in $\NC(G)$ implies that there exists an $\eta: \mathbb{S}^2 \to \NC(G)$ such that $\eta$ is homotopic to $ \alpha$ and the 2-simplex generated by $\{x,y,z\}$ lies on the boundary of its extension $\delta: \mathbb{D}^3 \to \NC(G)$. Therefore, we may assume that $\{x,y,z\} \in bd(\sigma) = \alpha$. We now can use similar arguments as in the proof of Theorem \ref{fundamental grp comparison} to complete the proof.
 		
 		To prove the second part, note that whenever two vertices of a graph can be connected by a path of even length, the two vertices belong to the same path connected component of $\NC(G)$. Therefore if there is a path of even length for every pair of two neighbours of $v$, then all three neighbours of $v$ belong to the same path connected component of $\NC(G-v)$, and hence the result.
 		
 		For the rest of the proof, let $\alpha$ be a closed non-nullhomotopic 1-chain in $\NC(G-v)$ containing $a,b,$ and $c$ as its 0-simplices. Let $\alpha = a u_1 u_2\dots u_{k_1} b v_1 v_2 \dots v_{k_2} c w_1 w_2 \dots w_{k_3} a$, where $k_1,k_2,k_3 \ge 0$. Since for some $x,y \in \N_G(v) = \{a,b,c\}$, the distance between $x$ and $y$ is more than 2, at least one of $k_1,k_2,k_3$ is greater than 0.  Suppose $\alpha$ is trivial in $\NC(G)$, then $\alpha_1 = a u_1 u_2\dots u_{k_1} b a$, $\alpha_2 = b v_1 v_2 \dots v_{k_2} c b$, and $\alpha_3 = c w_1 w_2 \dots w_{k_3} a c$ are trivial 1-chains as $\{a,b,c\}$ is a 2-simplex in $\NC(G)$. Let $a,b$ be such that $P = a p_1 p_2 \dots p_{2n-1} b$. If $n\ge 2$,  there does not exist any 2-simplex in $\NC(G)$ containing $a,b$. This implies that $\alpha_1$ cannot be contracted to a point in $\NC(G)$, a contradiction. Therefore,	$\pi_1(\NC(G-v)) \text{  is a subgroup of } \pi_1(\NC(G))$. 
 	\end{proof}

 	\begin{remark}  As before, we can show that  if the $n^{th}$  homotopy group of $\NC(G)$ for some $n\geq 4$ is trivial, then after  removing this degree $3$ vertex $c$ the $\pi_n(\NC(G-c)) =0$ for that $n$.  Let $\C^{(3)}$ denote the class of graphs with minimum degree $3$. Then for any non-bipartite graph $G \in \C^{(3)}$ satisfying hypothesis of Theorem \ref{higher homotopy grp comparison} (2), we can reduce the graph $G$ and obtain a smaller subgraph $G'$ for which the computation of $\NC(G')$ is easier as $G'$ is strictly smaller than $G$, and $\conn \NC(G')$ is at least as much as $\conn \NC(G)$.   
\end{remark}

\begin{remark} Removing degree $2$ and $3$ vertices is especially useful if they are obstructing the connectivity of $\NC(G)$ to be higher. 
	
For instance,  
 consider the graph $G := C_n \displaystyle\sqcup_{K_2} K_m$ for any $n\ge 5$ and $m \ge 4$. Then $G$ has a degree 2 vertex (coming from the $C_n$), say $v$. It can be checked that $\conn \NC(G) = 0$, whereas $\conn \NC(G-v) = m-3 > 0.$ And at the same time, $\chi(G-v) = \chi(G)$, improving the Lov\'{a}sz' bound. Thus, Theorem \ref{fundamental grp comparison} helps improve the lower bound on the chromatic number of $G$.  

Similarly a simple application of Theorem \ref{higher homotopy grp comparison} is as follows. Take pushout of  $K_3 \to K_m$ to get $ K_m \displaystyle\sqcup_{K_3} K_m'$ for $m \ge 6$. Let $a \in V(K_m), b,c \in V(K_m')$ be such that $a,b \notin V(K_m) \cap V(K_m')$ and $c \in V(K_m) \cap V(K_m')$. Take a new vertex $x$ with neighbours $a,b$. Now take another vertex, say $v$ such that $v $ is adjacent to $a,c$ and $x$. Let $G$ be this graph. We first note that $\chi(G-v) = \chi(G)$. Also $\NC(G)$ is 0-connected, whereas the topological connectivity of $\NC(G-v)$ is at least 2 (cf. \cite{CSORBA-chordal}).
\end{remark}	
 	
 	 However, it is  possible to have graphs $G$ with $\chi(G) = 4$ but $\chi(G-v) = 3$ and characterizing  such graphs is difficult.  In such cases it is not clear when this method will give us better bounds for $\chi(G)$ using inequality \eqref{lovasz conjecture}. This indicates that  a degree 4 version of the above theorem may not be very fruitful. 
 	
 	 Note that Theorem \ref{higher homotopy grp comparison} can also be applied to graphs with a degree 3 vertex having the said properties even if the minimum degree of $G$ is two or less, that is, this theorem applies to a general class of graphs.

\vspace{0.2cm}

\section{Reducing computational complexity for hom complex}\label{sec:Dn applications}

	Let  $A$, $B$ and $C$ be graphs and  $f : A \to B $ and $g: A \to C$ be  graph homomorphisms. In the category of graphs, the double mapping cylinder of height $n$  is defined as   $D_n= B \sqcup_f (A \times I_n) \sqcup_g C = (B \sqcup (A \times I_n) \sqcup C)/\sim$, where $\sim$ denotes the equivalence relation $f(a) \sim (a,n)$ and $g(a) \sim (a,0), \text{ and } [x][y] \in E(D_n)$ if there exists some $x_0 \in [x],\ y_0 \in [y]$ such that $\{x_0,y_0\} \in E(B) \cup E(A \times I_n) \cup E(C)$. We write $x$ to denote its equivalence class $[x]$.
	\begin{figure}[H]
		$$ \xymatrix{
			A \ar[r]^f \ar[d]_g	&	B\ar[d]_{j_1}\\
			C \ar[r]^{j_2}			&	D_n
		}$$
		\caption{Double mapping cylinder.}
		\label{htpy}
	\end{figure}
	Given any two graph homomorphisms $f,g$, we now examine the relation between the pushout object and the double mapping cylinder of $f$ and $g$.
 	
 	\begin{lemma} \label{chromatic no. comparison}
 		Let $A,B,C \in \cG$ be three loopless graphs with graph homomorphisms $f : A \to B$ and $g : A \to C$. Let $G \in \cG$ denote the pushout of $f$ and $g$, and $D_n$ be their double mapping cylinder of height $n$. If $G$ is a simple graph, then
		$$\chi(G) \ge \chi(D_n).$$
 	\end{lemma}

 	\begin{proof}
 		For a loopless graph $X \in \cG$, a $k$-colouring of $X$ is a graph homomorphism from $X \to K_k$.
 		Let $X$, $Y$ be two loopless graphs, and $h : X \to Y$ be a graph homomorphism. Then it is easy to note that a k-colouring of $h(X)$ can be extended to a k-colouring of $X$. 
 		
 		Let $\chi(G) = m$ and $c :  G \to K_m$ be a proper $m$-colouring of $G$. By definition of the pushout (quotient graph) $G$, $f(a)$ is identified with $g(a)$ for every $a \in V(A)$, therefore  $c|_{f(A)} = c|_{g(A)}$. Then $c|_{f(A)} = c|_{g(A)}$ gives a proper colouring of $f(A)$ and $g(A)$ and hence a proper colouring, say $\overline{c}$ of $A$. Define $c' : D_n \to K_m$ by letting $c'|_{B \cup C} = c|_{B \cup C}$ and  $c'|_{A \x \{i\}} = \overline{c}$, for  $i =1,2,\dots n-1$. Clearly $c'$ is a proper colouring of $D_n$, hence the claim.		
 	\end{proof}
 	
 	The inequality in the other direction is not true in general. 
 	For example, consider the following graphs as $A, B, C$.
 	
 	\vspace{.5cm}
 	
 	\begin{tabular}{m{1.6in} m{0.1in} m{1.6in} m{0.1in} m{1.6in}}
 		
		{\begin{center}
 			\begin{tikzpicture}[scale=.7]
 				\tikzstyle{edge} = [draw,thick,-]
 				\draw[edge] (0,2) -- (1,1) -- (3,1) -- (4,2);
 				\node [left] at (-0.1,2) { $a$};
 				\draw [fill] (0,2) circle [radius=0.15];
 				\node [below] at (1,.9) { $d$};
 				\draw [fill] (1,1) circle [radius=0.15];
 				\node [below] at (3,0.9) { $c$};
 				\draw [fill] (3,1) circle [radius=0.15];
 				\node [right] at (4.1,2) { $b$};
 				\draw [fill] (4,2) circle [radius=0.15];
			\end{tikzpicture}
			
 			Graph A
 			\end{center}
 		}
 		&	&
 		{\begin{center}
 				\begin{tikzpicture}[scale=0.7]
 				\tikzstyle{edge} = [draw,thick,-]	
 				\draw[edge] (0,2) -- (4,2) -- (3,1) -- (1,1) -- (0,2) -- (3,1);
 				\draw[edge] (0,2) -- (2,3) -- (4,2) -- (1,1);
 				\node [left] at (-0.1,2) { $a$};
 				\draw [fill] (0,2) circle [radius=0.15];
 				\node [below] at (1,.9) { $d$};
 				\draw [fill] (1,1) circle [radius=0.15];
 				\node [below] at (3,0.9) { $c$};
 				\draw [fill] (3,1) circle [radius=0.15];
 				\node [right] at (4.1,2) { $b$};
 				\draw [fill] (4,2) circle [radius=0.15];
 				\node [above] at (2,3.1) { $x$};
 				\draw [fill] (2,3) circle [radius=0.15];
 				 				
 			\end{tikzpicture}
 				
 				Graph B
 		\end{center}
 		}
 		&	&
 		{\begin{center}
 			\begin{tikzpicture}[scale=0.7]
 				\tikzstyle{edge} = [draw,thick,-]
 				\draw[edge] (4,2) -- (3,1) -- (1,1) -- (0,2);
 				\draw[edge] (1,1) -- (2,-1) -- (3,1);
 				\node [left] at (-0.1,2) { $a$};
 				\draw [fill] (0,2) circle [radius=0.15];
 				\node [above] at (1.2,1) { $d$};
 				\draw [fill] (1,1) circle [radius=0.15];
 				\node [above] at (2.8,1) { $c$};
 				\draw [fill] (3,1) circle [radius=0.15];
 				\node [right] at (4.1,2) { $b$};
 				\draw [fill] (4,2) circle [radius=0.15];
 				\node [below] at (2,-1.1) { $y$};
 				\draw [fill] (2,-1) circle [radius=0.15];
 				
 				\draw[edge] (0,2) to [out=225,in=180] (2,-1) to [out=0,in=315] (4,2);
 			\end{tikzpicture}
 				
 			Graph C
 		\end{center}
 		}
 	\end{tabular}

 	\begin{tabular}{m{3.5in}m{2in}}
 		{\hspace{-.2cm}
 			Let $f : A \to B$ and $g : A \to C$ be graph inclusions. Then the graph on the right is the pushout $G = B \displaystyle\sqcup_A C$ and has $\chi(G) = 5$. On the other hand, the double mapping cylinder $D_n = B \displaystyle\sqcup_f (A \x I_n) \displaystyle\sqcup_g C$ is 4-colourable for any $n>1$.
 		
 	 	This gives an example where $\chi(G)$ is strictly greater  than $\chi(D_n)$. }
 		&
 		{\begin{center}
 				\begin{tikzpicture}[scale=.68]
 					\tikzstyle{edge} = [draw,thick,-]
 					\draw[edge] (4,2) -- (3,1) -- (1,1) -- (0,2);
 					\draw[edge] (1,1) -- (2,-1) -- (3,1);
 					\node [left] at (-0.1,2) { $a$};
 					\draw [fill] (0,2) circle [radius=0.15];
 					\node [left] at (0.9,1) { $d$};
 					\draw [fill] (1,1) circle [radius=0.15];
 					\node [right] at (3.1,1) { $c$};
 					\draw [fill] (3,1) circle [radius=0.15];
 					\node [right] at (4.1,2) { $b$};
 					\draw [fill] (4,2) circle [radius=0.15];
 					\node [below] at (2,-1.1) { $y$};
 					\draw [fill] (2,-1) circle [radius=0.15];
 					\draw[edge] (0,2) -- (4,2) -- (3,1) -- (1,1) -- (0,2) -- (3,1);
 					\draw[edge] (0,2) -- (2,3) -- (4,2) -- (1,1);
 					\draw[edge] (0,2) to [out=225,in=180] (2,-1) to [out=0,in=315] (4,2);
 					\node [above] at (2,3.1) { $x$};
 					\draw [fill] (2,3) circle [radius=0.15];
 				\end{tikzpicture}

 				Graph $G$
 		\end{center}
 		}
 	\end{tabular}
 	
	The definition of double mapping cylinder of graphs is analogous to the one for topological spaces which we recall here.
 	
 	\begin{defn} \label{htpy pushout in Top}
 		Let $X,Y,Z$ be topological spaces with $p:X \to Y$ and $r:X \to Z$ continuous maps. The double mapping cylinder of $\{p,r\}$ is defined to be the quotient space $$ Y \displaystyle\bigsqcup_p (X \x [0,1]) \displaystyle\bigsqcup_r Z / \sim$$
 		where $\sim$ is an equivalence relation generated by $(x,1) \sim p(x)$ and $(x,0) \sim r(x)$. We  write $ Y \displaystyle\bigsqcup_X^h Z$ to mean the double mapping cylinder of the maps $\{p,r\}$. 
 	\end{defn}
 	
 Note that the double mapping cylinders give homotopy pushouts in the category of topological spaces. In particular, given two  homotopy equivalent pushout diagrams, their double mapping cylinders in spaces  are homotopy equivalent. 
     We have the following proposition about hom complexes of graphs in this context. 
 	\begin{prop} \label{general T main}
 		Let $T\in \cG$ and consider $ \Hom(T, \underline{ \ \ }) : \cG \to \text{Top}$.
 		Then for the diagram in Figure \ref{htpy} and $n \ge \diam(T) + 1$,
 		$$ \xymatrix{
 			\Hom(T,A)\ar[r]^{f_T} \ar[d]_{g_T}	& \Hom(T,B)\ar[d]^{j_1}\\
 			\Hom(T,C)\ar[r]_{j_2}				& \Hom(T,D_n)
 		}$$
 		is the homotopy pushout in Top, the category of topological spaces, that is, $\Hom(T,D_n)$ is the double mapping cylinder of maps $\{f_T,g_T\}$.
 	\end{prop}
 	
 	We can give a combinatorial proof of the above result by describing the polytopes in the hom complex, $\Hom(T,D_n)$, which we omit here for brevity. 
 	 Alternatively, we can prove this result using properties of pushouts of topological spaces and modifying ideas used in the proof of   \cite[Theorem 5.1]{matsushita-MappingCylinder} by Matsushita. 
	  However, he has made no reference or use of the above proposition in his article.  	

\vspace{0.25cm}

Here is a quick example of a graph which is a double mapping cylinder or smaller graphs and the above proposition helps us compute its hom complex. 
\vspace{0.25cm}

 	\noindent{\bf Application  1: } 
Given any two natural numbers $m,r \in \mathbb{N}$, we use Theorem \ref{general T main} to construct a graph $G$ such that $\chi(G) = m$, and $\conn \Hom(K_r,G) = 0$.
 	First let $r=2$ and consider the graph $D_n = B \sqcup_f (A \times I_n) \sqcup_g C$ with $B = K_m,\ A = K_2,\ C = K_m$; $m \ge 3$ and $f : A \rightarrow B$, $g : A \rightarrow C$ be inclusions. For $m=6$, $D_n$ is the graph: 
 	
%
%
%
%
%
%
%
%
%
 
 \begin{center}
 \centering
 \includegraphics[width=4.5in]{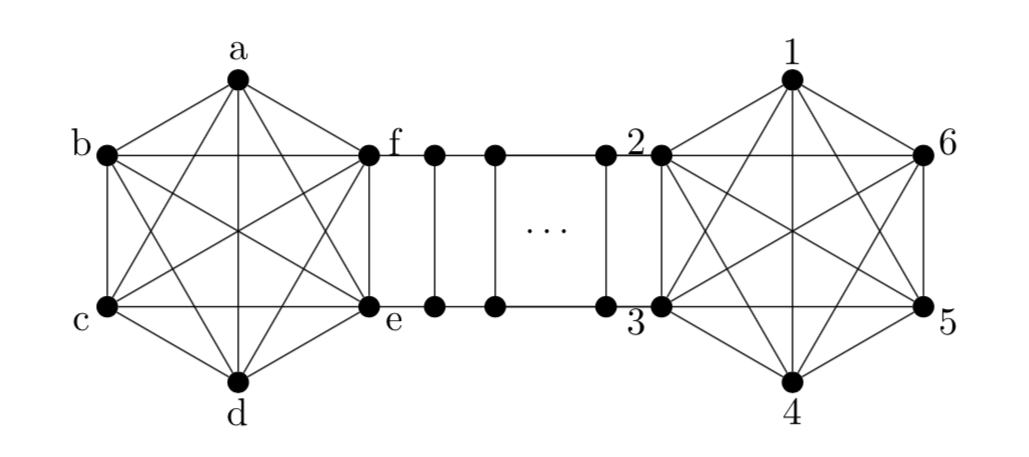}
 \end{center}

 	Since $f$ and $g$ are inclusions, the induced maps $f_{T} : \Hom(T,A) \rightarrow \Hom(T,B)$ and $g_T : \Hom(T,A) \rightarrow \Hom(T,C)$ are also inclusions. By Theorem \ref{general T main},
 	$\Hom(T,D_n) = \Hom(T,K_m \sqcup_f K_2 \x I_n \sqcup_g K_m) \simeq\Hom(T,B) \sqcup_{f_T} (\Hom(T,A) \x I) \sqcup_{g_T} \Hom(T,C)$, that is,
 	$$\Hom(K_2,D_n) \simeq\Hom(K_2,K_m) \sqcup_{f_{K_2}} (\Hom(K_2,K_2) \x I) \sqcup_{g_{K_2}} \Hom(K_2,K_m).$$
 	
 	The complex $\Hom(K_2, K_2) \x I$ is homotopy equivalent to two disjoint line segments, say $L_1$ and $L_2$. Suppose $L_1$ and $L_2$ have initial points $l_1,l_2$ and final points $l_1', l_2'$ respectively. Babson and Kozlov \cite{BabKoz-HOMcplx} proved  that $\Hom(K_2, K_m)$ is homotopy equivalent to a $(m-2)$-sphere $S^{m-2}$. 	
 	Thus for $T = K_2$, $f_T$ maps $l_1,l_2$ to $\Hom(K_2,B) \simeq S^{m-2}$ and $g_{T}$ maps $l_1',l_2'$ to $\Hom(K_2,C) \simeq  S^{m-2}$. 
 	
 	Therefore, $\Hom(K_2,D_n) \simeq S^{m-2} \vee S^1 \vee S^{m-2}$. In particular, $\conn \Hom(K_2,D_n) = 0$.
 	
 	For an arbitrary $r$ and $m$, $r< m$, define $G$ to be the double mapping cylinder of $f,g$ of height more than $r$, where $f=g : K_r \to K_m$ are the inclusion maps. Then the above argument shows that $\Hom(K_r,G)$ is homotopy equivalent to wedge of $1$-dimensional spheres, $\mathbb{S}^1$,  and $(m-r)$-dimensional spheres, $\mathbb{S}^{m-r}$. The graph $G$ is clearly $m$-colourable, and $\conn \Hom(K_r,G) = 0$.
 	
 	In fact, for any $r < p \le m$ if $G$ is the double mapping cylinder of $f : K_r \to K_m$, $g : K_r \to K_p$ of height more than $r$, then $\chi(G) = m$ and $\conn \Hom(K_r,G) = 0$. Let $\mathcal{T}$ be the class of graphs $G$ as defined above. Then by choosing $r$ small and $m$ large, $\mathcal{T}$ gives a class of graphs for which the homotopy test graph $K_r$ gives a weak lower bound for the chromatic number of graphs in $\mathcal{T}$ via the inequality \eqref{lovasz conjecture}.  \hfill{$\Box$}

 In the previous example, the graph $A$ is chosen to be $K_2$, which we recall is a test graph.	
 	We now generalise this construction for any homotopy test graph $T$. 
 	Let $G$ be a graph  and $v \in G$ be such that $G$ folds to $G-v$.
 	In view of \cite{fold-in-hom},  it is clear that for any graph $H$, 
	$\conn \Hom(G,H) = \conn \Hom(G-v,H)$.
 	Also, the chromatic number of $G$ is equal to the chromatic number of $G-v$.
 	Therefore, a graph $T$ is a homotopy test graph if and only if its unique (upto graph isomorphism) stiff subgraph $T'$ is a homotopy test graph. So without loss of generality let $T$ be a stiff homotopy test graph. We further assume that $T$ has an endomorphism\footnote{An endomorphism is a morphism from the object to itself.} (other than the identity morphism). Let $m \in \mathbb{N}$ be large enough so that there exists $i:T \to K_m$, a graph inclusion. Let  $n-1 = \diam(T)$, and $D_n$ be the double mapping cylinder of $f,g$ of length $n$, where $f=g=i$. Since $T$ is stiff, the identity morphism of $T$ is an isolated point in the hom complex of $\Hom(T,T)$ by Lemma 6.5 of \cite{x-htpy}. Therefore the hom complex, $\Hom(T,D_n)$, is homotopy equivalent to $X \vee \mathbb{S}^1$ for some prodsimplicial complex $X$ implying that $\pi_1(\Hom(T,D_n))$ is non-trivial, and hence $\conn \Hom(T,D_n) \le 0$. However, the chromatic number of $D_n$ is $m$.
 	We have shown the following:
 	
 	\begin{prop}\label{prop:test graph bad}
 		Let $T$ be a homotopy test graph with a non-identity endomorphism. For any $m \ge |V(T)|$, there exists a graph $G$ such that $\conn \Hom(T,G) = 0$, and $\chi(G) = m$.
 	\end{prop}
 		\vspace{.5cm}
 	
 	One can also use Theorem \ref{general T main} for computing homotopy type of several hom complexes easily if they are $\x$-homotopy equivalent to  double mapping cylinders of a pushout diagram of graphs. 
	 
	We next give such a quick computation for a class of graphs explored by Daneshpajouh in \cite{hamid-example}.

 	\noindent{\bf Application 2: }
 	Let $H$ and $L$ be  non-bipartite graphs  with disjoint vertex sets and let $x \in V(H)$ and $y \in V(L)$. Daneshpajouh defined $G_{H_x,L_y}$ to be the graph (cf. \cite[Definition 1]{hamid-example}) to be the graph with $V(G_{H_x,L_y}) = V(H) \cup V(L) \cup \{z\}$ and $E(G_{H_x,L_y}) = E(H) \cup E(L) \cup \{xz,zy\}$ where $z \notin (V(H) \cup V(L))$.
 	He computed the homotopy type of the neighbourhood complex of this graph. 
 	Using this construction, for any given positive integers $l,m$, and $2 \le p \le q$, he has defined a graph containing a copy of $K_{l,m}$, having the chromatic number, clique number and Lov\'{a}sz bound as $q,p$, and 3, respectively.

 	We compute the homotopy type of the $\NC(G_{H_x,L_y})$ using Theorem \ref{general T main}. 
 	Let $B$ and $C$ be two non-trivial graphs with special vertices $b \in V(B), c \in V(C)$. For $A = K_2$ with $V(K_2) = \{x,y\}$, let $B_b = (B \sqcup K_2)/(b \sim x), C_c = (C \sqcup K_2)/(c \sim x)$ be the graphs obtained by taking wedge sum (for definition, please refer to \cite{wedge-sum}) of $B$ with $K_2$ and $C$ with $K_2$ respectively. 
 	
 	Let $f : A \to B_b, g : A \to C_c$ be graph inclusions. Then the pushout of $f,g$ is the graph $G_{B_b,C_c}$.
 	Let $D_2$ be the double mapping cylinder of $f$ and $g$ of height 2. Then $D_2$ is $\x$-homotopic to $G_{B_b,C_c}$ via a sequence of folds, therefore by \cite{fold-in-hom}, $\NC(G_{B_b,C_c})$ is homotopically equivalent to $\NC(D_2)$.
 	By Theorem \ref{general T main}, $\NC(G_{B_b,C_c})$ is homotopically equivalent to the wedge $\NC(B_c) \vee \NC(C_c) \vee \mathbb{S}^1 \simeq \NC(B) \vee \NC(C) \vee \mathbb{S}^1$ as $B_b$ and $C_c$ folds to $B$ and $C$, respectively.

	 Proposition \ref{general T main} also applies in cases when the neighbourhood complex of the double mapping cylinder is  more connected than that of pushout.   
	For instance, consider the graph $G$ to be the wedge of $K_4$ and $K_3$ at a vertex, and choose the following as $A,B$ and $C$.
	
	\begin{tabular}{m{1.6in} m{0.2in} m{1.6in} m{0.2in} m{1.6in}}
		{\begin{center}
				\begin{tikzpicture}[scale=0.7]
				\tikzstyle{edge} = [draw,thick,-]		
				\draw[edge] (2,1.5) -- (0,0) -- (0,3) -- (2,1.5) -- (4,3);
				\draw[edge] (0,0) -- (.8,1.5) -- (0,3);
				\draw[edge] (.8,1.5) -- (2,1.5);
				\draw [fill] (0,0) circle [radius=0.15];
				\node [below] at (0,-0.1) { $x$};
				\draw [fill] (0,3) circle [radius=0.15];
				\node [above] at (0,3.1) { $y$};
				\draw [fill] (.8,1.5) circle [radius=0.15];
				\node [left] at (.7,1.5) { $z$};	
				\draw [fill] (2,1.5) circle [radius=0.15];	
				\node [below] at (1.9,1.3) { $u$};		
				\draw [fill] (4,3) circle [radius=0.15];		
				\node [above] at (4,3.1) { $v$};		
				\draw[fill](4,0)circle[radius=0.15];																							
				\node [below] at (4,-0.1) { $a$};
				\end{tikzpicture}
				
				Graph A
			\end{center}
		}
		&	&
		{\begin{center}
				\begin{tikzpicture}[scale=0.7]
				\tikzstyle{edge} = [draw,thick,-]		
				\draw[edge] (2,1.5) -- (0,0) -- (0,3) -- (2,1.5) -- (4,3) -- (4,0);
				\draw[edge] (0,0) -- (.8,1.5) -- (0,3);
				\draw[edge] (.8,1.5) -- (2,1.5);
				\draw [fill] (0,0) circle [radius=0.15];
				\node [below] at (0,-0.1) { $x$};
				\draw [fill] (0,3) circle [radius=0.15];
				\node [above] at (0,3.1) { $y$};
				\draw [fill] (.8,1.5) circle [radius=0.15];
				\node [left] at (.7,1.5) { $z$};	
				\draw [fill] (2,1.5) circle [radius=0.15];	
				\node [below] at (1.9,1.3) { $u$};		
				\draw [fill] (4,3) circle [radius=0.15];		
				\node [above] at (4,3.1) { $v$};		
				\draw[fill](4,0)circle[radius=0.15];																							
				\node [below] at (4,-0.1) { $a$};
				
				\end{tikzpicture}
				
				Graph B
			\end{center}
		}
		&	&
		{\begin{center}
				\begin{tikzpicture}[scale=0.7]
				\tikzstyle{edge} = [draw,thick,-] 				
				\draw[edge] (2,1.5) -- (0,0) -- (0,3) -- (2,1.5) -- (4,3);
				\draw[edge] (0,0) -- (.8,1.5) -- (0,3);
				\draw[edge] (.8,1.5) -- (2,1.5);
				\draw[edge] (2,1.5) -- (4,0);
				\draw [fill] (0,0) circle [radius=0.15];
				\node [below] at (0,-0.1) { $x$};
				\draw [fill] (0,3) circle [radius=0.15];
				\node [above] at (0,3.1) { $y$};
				\draw [fill] (.8,1.5) circle [radius=0.15];
				\node [left] at (.7,1.5) { $z$};	
				\draw [fill] (2,1.5) circle [radius=0.15];	
				\node [below] at (1.9,1.3) { $u$};		
				\draw [fill] (4,3) circle [radius=0.15];		
				\node [above] at (4,3.1) { $v$};		
				\draw[fill](4,0)circle[radius=0.15];																	
				\node [below] at (4,-0.1) { $a$};
				
				\end{tikzpicture}
				
				Graph C
			\end{center}
		}
	\end{tabular}
	
	\vspace{-0.5cm}
	
	With respect to inclusions of $A$ in $B$, and in $C$, let $D_n = B \displaystyle\bigsqcup_A^h C$, then computations show that $D_n \simeq_{\x} K_4$ and hence $\Hom(K_2,D_n)$ is homotopically equivalent to a 2-sphere, $\mathbb{S}^2$. 
	Therefore $\NC(D_n)$ is 1-connected. 
	We can compute $\NC(G)$ and see that its first homology is non trivial and hence is not 1-connected. Thus connectivity of $\NC(D_n)$ is higher in this pushout. Since $\chi(G)\geq \chi(D_n)$, here the Lovasz bound is improved. 
	
	However,  $\conn \NC(G)$ is not always smaller than  $\conn \NC(D_n)$.  For instance, in the first application  where we use Prop 3.3 to compute the connectivity of a double mapping cylinder, the neighbourhood complex of the pushout graph $K_6\displaystyle\coprod_{K_2} K_6$ has connectivity at least $1$ (cf. \cite{CSORBA-chordal}).

	\section{ Is $D_n$ a homotopy pushout in graphs?} \label{sec:Dn not good notion section}

	We have already seen several properties of the double mapping cylinder $D_n$. In order to fully utilize these ideas, we need a way to recognize when a given pushout diagram is weakly equivalent to a double mapping cylinder.
	
	The definition  of $D_n$ is analogous to that of double mapping cylinder in spaces. In topological spaces, the double mapping cylinder gives a homotopy pushout. Recall  homotopy pushout in spaces are preserved under  homotopy equivalent diagrams and we can recognize when a given pushout is a homotopy pushout (namely, if one of the maps in the pushout diagram  is a cofibration).
	
	Consider the category, $\cG$, of graphs with weak equivalences as $\x$-homotopy equivalences. If the  double mapping cylinder construction in  graphs is a homotopy pushout in the category of graphs, then we can build the theory around it to get a recognizing principle. 
	
		However, the double mapping cylinder fails to preserve $\x$-homotopy equivalent diagrams. Let $A \xrightarrow[h_A]{\simeq} A',\ B \xrightarrow[h_B]{\simeq} B',\ C \xrightarrow[h_C]{\simeq} C'$ be $\times$-homotopy equivalences with $\times$-homotopy inverses $A' \xrightarrow[h_{A'}]{\simeq} A,\ B' \xrightarrow[h_{B'}]{\simeq} B,\ C' \xrightarrow[h_{C'}]{\simeq} C$ respectively. Let $f' : A' \to B', g' : A' \to C'$ be graph homomorphisms. 
		There exists a canonical graph homomorphism $F : D_n \to D'_n$ which is defined as  $F(b) = h_B(b)$,  $F(a,i) = (h_A(a),i)$  and $F(c) = h_C(c)$ for all  $b \in V(B),\ a \in V(A), \  i \in V(I_n),\  c \in V(C)$.

		Let $f h_{A'} \underset{H_B}{\simeq_{\x}} h_{B'} f' $, $h_{C'} g' \underset{H_C}{\simeq_{\x}} g h_{A'}$ be $\times$-homotopies of length $m$ and $k$ respectively. Analogous to the case of topological spaces, we expect to define $F' : D_{n'}' \to D_n$ using the homotopies $H_C$ and $H_B$. 
		To do this in graphs, one needs to take the double mapping cylinder of $\{f',g'\}$ of length at least $m+k$ more than that of double mapping cylinder of $\{f,g\}$. Thus there is no  graph homomorphism  $F' : D_n' \to D_n$.
		
		Since the domain and codomain for the composite $F F'$ are double mapping cylinders of $\{f',g'\}$ of different lengths, $F F'$ cannot be $\x$-homotopic to the identity homomorphism of any graph. 		 
		As  mentioned earlier, $D_n$ is not $\x$-homotopic to $D_m$, for $n \ne m$. Theorem \ref{general T main} implies that there is a homotopy equivalence between  $\Hom(T,D_n)$ and $\Hom(T,D_m)$ for all $T$. However, this is not induced by a graph homomorphism. 
		
		We then check if  $F F'$ is $\x$-homotopic to a graph homomorphism, 'shrink', that induces a homotopy equivalence on $\Hom(T,\underline{\ \ })$,  for all $T$. If this turned out to be true, then we would add a few more weak equivalences to our category  of graphs without compromising the structure. In the category of graphs with this new class  of weak equivalences, we will stand a chance at characterizing $D_n$ as a homotopy pushout. However, this is also false.
		
		We define the graph homomorphism $shrink_{m+k} : D_{m+n+k}' \rightarrow D_{n}'$ as identity on $ B' \sqcup C'$ and
		\begin{equation}
		shrink_{m+k}(a',s)=
		\left\{
		\begin{array}{l l l}	
		(a',0)		&	\mbox{if } 0 \leq s \leq k, \\
		(a', s-k)	&	\mbox{if } k \leq s \leq k+n, \\
		(a', n)		&	\mbox{if } k+n \leq s \leq k+n+m,
		\end{array}
		\right.
		\end{equation}
		for all $a' \in V(A') $ and $ s \in V(I_{m+n+k})$. The subscript $m+k$ in  $shrink_{m+k}$ indicates the levels above and below $n$, respectively, that the map $shrink_{m+k}$ shrinks.
		Clearly, $shrink_{m+k}$ induces a homotopy equivalence from $\Hom(T,D_{m+n+k}') \to \Hom(T,D_{n}')$, for all $T \in \cG$ and $m,n,k \ge 0$.
		
		Let $l = \max\{m,k\}$ so that all the given homotopies are of the same length $l$. This can be done by repeating the values on the last level of the homotopy on the newly added levels.

		For $n' = n+ 2l$, $F' : D_{l + n + l}' \rightarrow D_n$ can be defined as
		$F'(b') = h_{B'}(b')$,	$F'(c') = h_{C'}(c')$, for all $b' \in B'$ and $ c' \in C'$,
		and for $a' \in A', s \in I_{l + n + l}$, 
		$$F'(a',s) = 
		\left\{
		\begin{array}{l l l}
		H_C(a',s) & \mbox{if } 0 \leq s \leq l,\\
		(h_{A'}(a'),s-l) & \mbox{if } l \leq s \leq l+n,\\
		H_B(a', s-(n+l)) & \mbox{if } l+n \leq s \leq l+n+l,
		\end{array}\right.
		$$

		Suppose $F F' \simeq_{\x} shrink_{l+l}$, then for some $r \in \mathbb{N}$ there exists a $\x$-homotopy $K : D'_{l+n+l} \x I_r \to D'_n$ such that $K|_{D'_{l+n+l} \x \{0\}} = F F'$ and $K|_{D'_{l+n+l} \x \{r\}} = shrink_{l+l}$.
		Since $F F'|_{A' \x I_l}$ is also a $\x$-homotopy, in general $ K(a',l-1,0) = F F' (a',l-1) = h_C H_C (a',l-1)$ need not be adjacent to $K(a'',l,1)$, where $a'a'' \in E(A')$. 
		Therefore, with no other assumptions, if we use only these given homotopies, then there does not exist a $\x$-homotopy between $FF'$ and the $shrink_{l+l}$. 
		
		 One possible approach to resolve this issue is  to define an appropriate class of cofibrations with  the class of weak equivalences as  the $\x$-homotopy equivalences. Such a setup will allow us to give a better definition of homotopy pushouts in graphs. We explore this  by studying model structures on the category of finite graphs with $\x$-homotopy equivalences  in \cite{self-model-category}.


 \end{document}